\def\draft{n}
\documentclass{amsart}
\usepackage[headings]{fullpage}
\usepackage{amssymb,epic,eepic,epsfig}
\usepackage{pb-diagram,lamsarrow,pb-lams}

\theoremstyle{plain}

\newtheorem{theorem}{Theorem}
\newtheorem{proposition}{Proposition}[section]

\newtheorem{corollary}[proposition]{Corollary}

\newtheorem{conjecture}{Conjecture}

\theoremstyle{definition}
\newtheorem{definition}[proposition]{Definition}
\newtheorem{problem}{Problem}

\theoremstyle{remark}
\newtheorem{example}[proposition]{Example}

\newtheorem{remark}[proposition]{Remark}

\def\printname#1{
        \if\draft y
                \smash{\makebox[0pt]{\hspace{-0.5in}
                        \raisebox{8pt}{\tt\tiny #1}}}
        \fi
}

\newcommand{\psdraw}[2]
         {\begin{array}{c} \hspace{-1.3mm}
        \raisebox{-4pt}{\epsfig{figure=draws/#1.eps,width=#2}}
        \hspace{-1.9mm}\end{array}}

\newlength{\standardunitlength}
\catcode`\@=11
\long\def\@makecaption#1#2{%
     \vskip 10pt

\setbox\@tempboxa\hbox{
       \small\sf{\bfcaptionfont #1. }\ignorespaces #2}%
     \ifdim \wd\@tempboxa >\captionwidth {%
         \rightskip=\@captionmargin\leftskip=\@captionmargin
         \unhbox\@tempboxa\par}%
       \else
         \hbox to\hsize{\hfil\box\@tempboxa\hfil}%
     \fi}
\font\bfcaptionfont=cmssbx10 scaled \magstephalf
\newdimen\@captionmargin\@captionmargin=2\parindent
\newdimen\captionwidth\captionwidth=\hsize
\catcode`\@=12

\newcommand{\tr}{\operatorname{tr}}

\def\lbl#1{\label{#1}\printname{#1}}

\def\BN{\mathbb N}
\def\BZ{\mathbb Z}

\def\BQ{\mathbb Q}
\def\BR{\mathbb R}
\def\BC{\mathbb C}

\def\D{\Delta}

\def\calP{\mathcal P}

\def\a{\alpha}

\def\La{\Lambda}
\def\l{\lambda}

\def\S{\Sigma}

\def\w{\omega}

\def\e{\epsilon}

\def\b{\beta}

\def\longto{\longrightarrow}
\def\w{\omega}

\def\SL{\mathrm{SL}}

\def\Om{\Omega}

\def\calB{\mathcal{B}}

\def\CS{\mathrm{CS}}
\def\Vol{\mathrm{Vol}}

\def\w{\omega}

\def\Om{\Omega}

\def\Li{\mathrm{Li}}

\def\calB{\mathcal{B}}

\def\calA{\mathcal{A}}
\def\fg{\mathfrak{g}}
\def\CS{\mathrm{CS}}
\def\Hom{\mathrm{Hom}}
\def\Lnp{L^{\mathrm{np}}}
\def\Lp{L^{\mathrm{p}}}
\def\SL{\mathrm{SL}}
\def\SU{\mathrm{SU}}
\def\calK{\mathcal{K}}
\def\ZLMO{Z^{\mathrm{LMO}}}

\def\ft{\mathrm{ft}}
\def\ARES{\mathrm{ARES}}
\def\RES{\mathrm{RES}}

\def\bterm{balanced term}
\def\gterm{general term}
\def\ft{\mathfrak{t}}

\def\GM{\mathrm{GM}}
\def\eLambda{e\Lambda}
\def\SP{\Sigma\Pi}

\begin{document}


\title[Chern-Simons theory, analytic continuation and arithmetic]{Chern-Simons 
theory, analytic continuation and arithmetic}
\author{Stavros Garoufalidis}
\address{School of Mathematics \\
         Georgia Institute of Technology \\
         Atlanta, GA 30332-0160, USA \\ 
         {\tt http://www.math.gatech} \newline {\tt .edu/$\sim$stavros } }
\email{stavros@math.gatech.edu}

\thanks{The author was supported in part by NSF. \\
\newline
1991 {\em Mathematics Classification.} Primary 57N10. Secondary 57M25.
}

\date{October 27, 2008 }


\begin{abstract}
The purpose of the paper is to introduce some conjectures regarding the
analytic continuation and the arithmetic properties of quantum invariants
of knotted objects. More precisely, we package the perturbative and
nonperturbative invariants of knots and 3-manifolds into two power series
of type P and NP, convergent in a neighborhood of zero, 
and we postulate their arithmetic resurgence.
By the latter term, we mean analytic continuation as a 
multivalued analytic function in the complex numbers minus a discrete set of 
points, with restricted singularities, local and global monodromy. 
We point out some key features of arithmetic resurgence in connection to 
various problems of asymptotic expansions of exact and perturbative 
Chern-Simons theory with compact or complex gauge group.
Finally, we discuss theoretical and experimental evidence for our conjecture.
\end{abstract}

\maketitle

\tableofcontents

\section{Introduction}
\lbl{sec.intro}


\subsection{Chern-Simons theory and analytic continuation}
\lbl{sub.analyticCS}

Chern-Simons Quantum Field Theory in 3-dimensions (perturbative, or
non-perturbative) produces a plethora of numerical invariants of knotted
3-dimensional objects. We introduce a packaging of these invariants into
two power series: one that encodes non-perturbative invariants (model
NP), and one that encodes perturbative invariants (model P).  
The paper is concerned with the analytic continuation, the asymptotic behavior
and and arithmetic properties of those power series.

Let us begin with a conjecture concerning the structure of nonperturbative
quantum invariants. Consider the generating series

\begin{equation}
\lbl{eq.Fzz}
\Lnp_{M,G}(z)=\sum_{n=0}^\infty Z_{M,G,n} z^n
\end{equation}
of the {\em Witten-Reshetikhin-Turaev} invariants $Z_{M,G,n}$ (see Section
\ref{sec.cs}) of a closed, oriented, connected 3-manifold $M$, 
using a compact Lie group $G$ and a level 
$n \in \BN$. 
The power series \eqref{eq.Fzz} is known to be convergent inside the unit 
disk $|z|<1$ since unitarity implies that $Z_{M,G,n}$ grows at most
polynomially with respect to $n$; see \cite{Ga1}.

\begin{conjecture}
\lbl{conj.0}(Analytic Continuation)
For every pair $(M,G)$ as above, the series $\Lnp_{M,G}(z)$ has analytic 
continuation as a multivalued function on $\BC\setminus\eLambda_{M,G}$,
where $\eLambda_{M,G} \subset \BC$ is a finite set that contains zero and
the exponentials of
the negative of the critical values of the complexified Chern-Simons action.
\end{conjecture}
A key observation is that 
$\eLambda_{M,G}$ may contain elements inside the unit disk $|z| <1$ despite
the fact that the power series $\Lnp_{M,G}(z)$ is analytic for $z$ such
that $|z| <1$. One may compare this behavior 
with the power series $\sum_{n=1}^\infty z^n/n^2$ that define the 
classical dilogarithm, whose analytic continuation is a multivalued
analytic function in $\BC\setminus\{0,1\}$.
Schematically, the analytic continuation of $\Lnp_{M,G}(z)$
may be depicted as follows:

$$
\psdraw{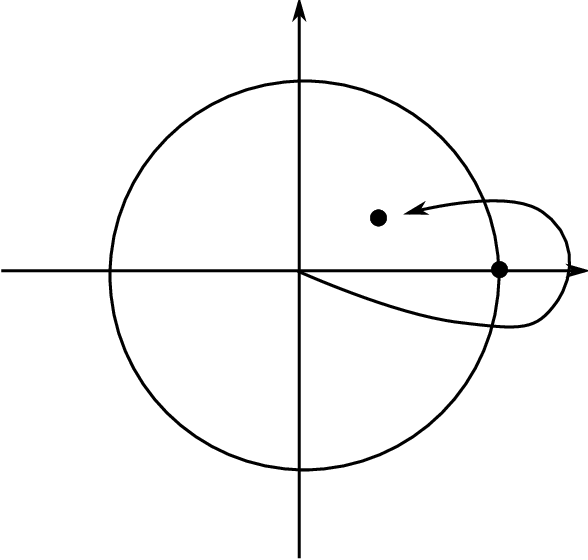}{2in}
$$

The above conjecture has the following features:

\begin{itemize}
\item[(a)]
It can be formulated for pairs $(\calK,G)$ where
$\calK$ denotes a knotted object, i.e., a knot $K$ in 3-space or
a closed 3-manifold $M$ and $G$ denotes a compact Lie group. 
\item[(b)]
It implies via elementary complex analysis two well-known Asymptotic 
Conjectures in Quantum Topology;
namely the Volume Conjecture (in the case of knots), and the Witten
Conjecture (in the case of 3-manifolds). The complex analysis argument
uses the Cauchy formula to write $Z_{M,G,n}$ as a contour integral 
of $\Lnp_{M,B}(z)/z^{n+1}$ and then deform the contour around the
singularities of the integrand nearest to the origin. For a detailed 
discussion, see \cite[Sec.7]{CG1}.
\item[(c)] 
It states a precise relation between exact Chern-Simons theory
and its perturbation expansion around a trivial (or not) flat connection.
Namely, perturbation theory is simply the expansion of
the multivalued function $\Lnp_{M,G}(z)$ around one of its singularities.
\item[(d)] 
It explains the effect of complexifying a compact gauge group
and to the partition function of the corresponding gauge theory.
Indeed, analytic continuation captures the critical values 
of the complexified action; compare also with \cite{GM,Vo}.
\item[(e)]
The Conjecture can be extended to state-sum invariants of sum-product type
that generalize $Z_{M,G,n}$ and do not necessarily come from topology.
\item[(f)]
The Conjecture has been proven for power series of 1-dimensional sum-product
type, which includes the case of the $3_1$ and $4_1$ knots; see \cite{ES}
and \cite{CG2}.
\item[(g)]
The Conjecture can, and has been, numerically tested. See Section 
\ref{sec.evidence}.
\end{itemize}

\subsection{Chern-Simons theory and Symmetry}
\lbl{sub.symmetry}

Our next conjecture is a Symmetry Conjecture. Recall that $M$ denotes an
oriented 3-manifold; we let $\tau M$ denote the {\em orientation reversed} 
manifold. 

\begin{conjecture}
\lbl{conj.symmetry}(Symmetry)
For every pair $(M,G)$ with $M$ an integer homology sphere, we have:
\begin{equation}
\lbl{eq.symmetry}
L_{M,G}(z):=\Lnp_{\tau M,G}(z)-\Lnp_{M,G}(1/z)
\end{equation}
has singularities at $z=0,1,\infty$.
\end{conjecture}

Let us make some comments regarding the above {\em paradoxical statement}:

\begin{itemize}
\item[(a)]
the left (resp. right) hand side is given by a convergent power series for 
$|z|<1$ (resp. $|z|>1$). Thus, the power series never
make sense simultaneously, but their analytic continuations do.
\item[(b)]
Zagier calls a similar statement in \cite[Eqn.7]{Za1} a 
{\em strange identity} since the the two sides never make sense 
simultaneously. Our Symmetry Conjecture is closely related
to a {\em modular property},
at least for the series studied by Kontsevich-Zagier; see \cite[Sec.6]{Za1}.
\item[(c)]
In physics, Equation \eqref{eq.symmetry} is usually called a {\em duality}.
\item[(d)]
In algebraic geometry and number theory, one may compare \eqref{eq.symmetry} 
with the following symmetry for the polylogarithm:
\begin{equation}
\lbl{eq.symmetrypoly}
\Li_k(z)+(-1)^k \Li_k(1/z)=-\frac{(2\pi i)^k}{k!} 
B_k\left(\frac{\log(z)}{2\pi i}\right)
\end{equation}
where $\Li_k(z)=\sum_{n=1}^\infty z^n/n^k$ is the $k$-th polylogarithm
and $B_k(z)$ is the $k$-th Bernoulli polynomial; see \cite[Sec.1.3]{Oe}.
\item[(e)]   
In analysis one may use the above symmetry to deduce the asymptotic behavior 
(and even more, the asymptotic expansion) of $\Lnp_{M,G}(z)$ for large $|z|$.
In particular, if $L_{M,G}(z)=0$, it follows that
\begin{equation}
\lbl{eq.largez}
\Lnp_{M,G}(z)=1+O\left(\frac{1}{z}\right)
\end{equation}
for large $|z|$.
\item[(f)]
The above symmetry may be explained by the fact that CS changes sign under
   orientation reversal. Since the level is nonnegative, the path integral
   formula for $\Lnp_{M,G}(z)$ formally implies the above symmetry.
\item[(g)]
If we use the normalized invariants
\begin{equation}
\lbl{eq.Fzznorm}
\hat\Lnp_{M,G}(z)=\sum_{n=0}^\infty \hat{Z}_{M,G,n} z^n, \qquad
\hat{Z}_{M,G,n}= \frac{Z_{M,G,n}}{Z_{S^3,G,n}}
\end{equation}
and if $M$ is an integer homology sphere, then it is possible that
\begin{equation}
\lbl{eq.possible}
\hat\Lnp_{\tau M,G}(z)-\hat\Lnp_{M,G}(1/z)=0.
\end{equation}
\item[(h)]
When $L_{M,G}(z)=0$, it follows that the asymptotic expansion of $\Lnp_{M,G}$ 
around its 
   singularities {\em uniquely determines} $\Lnp_{M,G}$. Indeed, the difference
   between two determinations is an entire function which is bounded by a 
   constant by the Symmetry Conjecture. Thus, the difference is identically 
   zero.
\item[(i)]
If $M=M$ is {\em amphicheiral} (for example, $M$ is given by a 
connected sum $M=N \# \tau N $ where $\tau$ is the orientation reversing
involution), then Equation \eqref{eq.possible} 
predicts that $\hat\Lnp_{M,G}(z)=\hat\Lnp_{\tau M,G}(1/z)$.
\end{itemize}

\subsection{Chern-Simons theory and P versus NP}
\lbl{sub.PNP}

So far, we considered nonperturbative quantum invariants. Let us now consider
perturbative quantum invariants of pairs $(\calK,G)$. They can be packaged
into a power series $\Lp_{M,G}(z)$ which is convergent at $z=0$. For 
a detailed definition, see Section \ref{sec.cs}.

Our next conjecture describes an explicit relation between 
the perturbative and nonperturbative quantum invariants.

\begin{conjecture}
\lbl{conj.2}(Exact Implies Perturbative)
For every integer homology sphere $M$
we have:
\begin{equation}
\lbl{eq.delta0}
\hat\Lnp_{M,G}(1+z)=\log(z) \Lp_{M,G}(\log(1+z))+h(z) 
\end{equation}
where $h(z)$ is a holomorphic function of $z$ at $z=0$ and 
$\hat\Lnp_{M,G}(1+z)$ denotes the analytic continuation at $1+z$ along {\em
any} path that avoids the singularities.
\end{conjecture}

As before, we can extend Conjecture \ref{conj.2} to pairs $(K,\SU(2))$,
where $K$ is a knot in 3-space.

\begin{remark}
\lbl{rem.alien}
In \'Ecalle's terminology (see \cite{Ec1} and also 
\cite[Sec.2.3]{Sa}), if $\D_z$ denotes the {\em alien derivative}
in the direction $z$, Conjecture \ref{conj.2} states that:
\begin{eqnarray}
\lbl{eq.delta1}
\Delta_1\hat\Lnp_{M,G}(z) &=& \Lp_{M,G}(\log(1+z)) \\
\lbl{eq.delta2}
\Delta_{1-z}\Delta_z\hat\Lnp_{M,G}(z) &=& 
0, \qquad \text{for} \qquad z \in \BC\setminus\{0,1\}.
\end{eqnarray}
Equation \eqref{eq.delta2} is reminiscent of the condition $z \wedge (1-z)
\in \wedge^2(\BC^*)$ that defines the {\em Bloch group}; 
see for example \cite{Ga4}.
\end{remark}

\begin{remark}
\lbl{rem.Lpreverse}
The series $\Lp_{\calK,G}(z)$ also satisfies a Symmetry Property:
\begin{equation}
\lbl{eq.symmetryp}
\Lp_{\tau\calK,G}(z)=\Lp_{\calK,G}(-z).
\end{equation}
Unlike the case of Conjecture \ref{conj.symmetry}, this is an easy corollary 
of its very definition.
\end{remark}

\begin{remark}
\lbl{rem.gukov}
Chern-Simons theory with complex gauge group was studied extensively by
Gukov in \cite{Gu}. It is an interesting problem to compare forthcoming 
work of Gukov-Zagier on modularity properties of the quantum invariants
with our conjectures.   
\end{remark}

\subsection{Chern-Simons theory and arithmetic resurgence}
\lbl{sub.CSari}

Based on some partial results of \cite{ES} and \cite{CG2} and 
stimulating conversations with O. Costin, J. \'Ecalle and D. Zagier, more is
actually expected to be true. Namely, we expect arithmetic restrictions
on the singularities of the series $\Lnp_{M,G}(z)$ and of its monodromy, 
local and global. These restrictions lead us naturally to the notion of 
arithmetic resurgence, and the Gevrey series of mixed type. 
In the rest of the paper, we will formulate these
expected algebraic/arithmetic aspects of quantum invariants in a precise way
and to expose the reader to the wonderful world of resurgence, introduced
by \'Ecalle in the eighties for unrelated reasons; \cite{Ec1}.

The logical dependence of the sections is the following:

$$
\divide\dgARROWLENGTH by2
\begin{diagram}
\node[2]{\text{Section \ref{sec.cs}}}
\arrow{sw}\arrow{se}
\node[2]{\text{Section \ref{sec.ar}}}
\arrow{sw}\arrow{se} 
\\
\node{\text{Section \ref{sec.habiro}}}
\arrow{see}
\node[2]{\text{Section \ref{sec.conj}}}
\arrow{s}
\node[2]{\text{Section \ref{sec.SP}}}
\arrow{sww}
\\
\node[3]{\text{Section \ref{sec.evidence}}}
\end{diagram}
$$

\section{Chern-Simons theory and invariants of knotted objects}
\lbl{sec.cs}


\subsection{Model NP: Non-perturbative invariants of 3-manifolds}
\lbl{sub.npinvariants}

In this section $M$ will denote a closed 3-manifold and $G$ will denote
a simple, compact, simply connected group $G$. For example, $G=\SU(2)$.

The Witten-Reshetikhin-Turaev invariant is a map:

\begin{equation}
\lbl{eq.WRT}
Z_{M,G}: \BN \longto \BC.
\end{equation}
For a definition of $Z_{M,G,n}$ see \cite{RT,Tu2,Wi}.
Formally, for $n \in \BN$, $Z_{M,G,n}$ is the expectation value of a
{\em path integral} with a topological {\em Chern-Simons Lagrangian} 
at level $n$; see \cite{Wi}.
Since the Chern-Simons Lagrangian takes values in $\BR/\BZ$,
it follows that the level $n$ has to be an integer number, which without loss
we take it to be nonnegative.
We can convert the sequence $(Z_{M,G,n})$ into a generating series as follows:

\begin{definition}
\lbl{def.Lnp}
For every $M$ and $G$ as above, we define:
\begin{equation}
\lbl{eq.Fz}
\Lnp_{M,G}(z)=\sum_{n=0}^\infty Z_{M,G,n} z^n
\end{equation}
\end{definition}

Unitarity of the Chern-Simons theory implies that for every $M,G$ the sequence
$(Z_{M,G,n})$ grows polynomially with respect to $n$. In other words, it
was shown in \cite{Ga1} that there exists positive constant $C$ and 
$m \in \BN$ (that depend on $M$ and $G$) so that
$$
|Z_{M,G,n}| < C n^m
$$
for all $n \in \BN$. Thus, $\Lnp_{M,G}(z)$ is analytic inside the unit disk
$|z| < 1$.

\subsection{Model P: Perturbative invariants of 3-manifolds}
\lbl{sub.pinvariants}

The path integral interpretation of $Z_{M,G,n}$ formally leads to a 
perturbation
theory along a distinguished critical point of the Chern-Simons action,
namely the trivial flat connection. This gives rise to a graph-valued
power series invariant, which has been defined by Le-Murakami-Ohtsuki in 
\cite{LMO}. Additional definitions of this powerful invariant were given
by Kuperberg-Thurston; see \cite{KT}.
More precisely, LMO define a graph-valued invariant $\ZLMO_M \in 
\calA(\emptyset)$ where $\calA(\emptyset)$ is a completed vector space of 
Jacobi diagrams. A {\em Jacobi} diagram of degree $n$ is a trivalent graph 
with $2n$ oriented vertices, considered modulo the AS and IHX relations; 
see \cite{B-N}. Jacobi diagrams are diagrammatic
analogues of tensors on a Lie algebra with an invariant inner product.
Indeed, given a simple Lie algebra $\fg$, there is a {\em weight system map}
that replaces a Jacobi diagram of degree $n$ by a rational number times 
$x^{-n}$: 
$$
W_{\fg}: \calA(\emptyset) \longto \BQ[[1/x]]
$$
see \cite{B-N}. Recall the {\em Borel transform}:
\begin{equation}
\lbl{eq.borel}
\calB: \BQ[[1/x]] \longto \BQ[[z]], \qquad
\calB\left(\sum_{n=0}^\infty \frac{a_n}{x^n} \right)
=\sum_{n=0}^\infty \frac{a_{n+1}}{n!} z^n
\end{equation}

\begin{definition}
\lbl{def.Lp}
Let $\Lp_{M,\fg}(z)$ denote the Borel transform of $W_{\fg} \circ \ZLMO_M$.
\end{definition} 

In \cite{GL2} it was proven that if $M$ is a homology sphere and $\fg$ is a 
simple 
Lie algebra, then the formal power series $W_{\fg} \circ \ZLMO_M$ is Gevrey-1.
In other words, $\Lp_{M,\fg}(z)$ is an analytic function in a neighborhood 
of $z=0$.

\subsection{The critical values of the Chern-Simons action and the dilogarithm}
\lbl{sub.geometricM}

Our main Resurgence Conjecture \ref{conj.1} formulated in Section 
\ref{sec.conj} below links the singularities of the analytic continuation of 
the series $\Lnp_{\calK,G}(z)$
and $\Lp_{\calK,G}(z)$ to some classical geometric invariants of 3-manifolds,
namely the critical values of the complexified Chern-Simons function. Let
us recall those briefly, and refer the reader to \cite{GZ,Wi,Ne,Ga4} for a 
more detailed discussion.

Let us fix a closed 3-manifold $M$, and simple, compact simply connected group
$G$, and a trivial bundle $M \times G$ with the trivial connection $d$.
Let $\calA$ denote the set of $G$-connections on $M \times G$. There is a 
Chern-Simons map:

\begin{equation}
\lbl{eq.CSA}
\CS: \calA \longto \BR/\BZ(2)
\end{equation}
where, as common in algebraic geometry, we denote

\begin{equation}
\lbl{eq.BZn}
\BZ(n)=(2 \pi i)^n\BZ.
\end{equation}
Even though $\calA$ is an affine infinite dimensional vector space acted
on by an infinite dimensional gauge group, the set $X_G(M)$ of gauge 
equivalence classes of the critical points of $\CS$ is a compact semialgebraic
set that consists of flat $G$-connections. Up to gauge equivalence,
the latter are determined by
their monodromy. In other words, we may identify:

\begin{equation}
\lbl{eq.XG}
X_G(M)=\Hom(\pi_1(M), G)/G
\end{equation}
This gives rise to a map:

\begin{equation}
\lbl{eq.CS}
\CS: X_G(M) \longto \BR/\BZ(2), \qquad A \mapsto 
\CS(A)=\int_M \tr(A\wedge dA+\frac{2}{3} A \wedge A \wedge A).
\end{equation}
Stokes's theorem implies that $\CS$ is a locally constant map. Since
$X_G(M)$ is a compact set, $\CS$ takes finitely many values in 
$\BR/\BZ(2)$. Let us now {\em complexify} the action; see also \cite{Vo}.
This means that we replace the compact Lie group $G$ by its
complexification $G_{\BC}$, the moduli space $X_G(M)$ by $X_{G_{\BC}}(M)$,
and the Chern-Simons action $\CS$ by $\CS_{\BC}$:

\begin{equation}
\lbl{eq.CSC}
\CS_{\BC}: X_{G_{\BC}}(M) \longto \BC/\BZ(2)
\end{equation}
$\CS_{\BC}$ is again a locally constant map, and takes finitely many values
in $\BC/\BZ(2)$.
Thus, we may define the following geometric invariants of 3-manifolds.

\begin{definition}
\lbl{def.singN}
For $M$ and $G$ as above, we define
\begin{eqnarray}
\lbl{eq.singNp}
\Lambda_{M,G} &=& \cup_{\rho \in X_{G_{\BC}}(M)} 
(-\CS_{\BC}(\rho) + \BZ(2) ) \subset \BC, \qquad \eLambda_{M,G}=\{0\}\cup
\exp\left(\frac{1}{2 \pi i} \Lambda_{M,G} \right) \subset \BC.
\end{eqnarray}
\end{definition}

\begin{remark}
\lbl{rem.eLambda}
Under complex conjugation (but keeping the orientation of the ambient 
manifold fixed), we have $\CS_{\BC}(\bar{\rho})=\overline{\CS_{\BC}(\rho)}$.
It follows that $\Lambda_{M,G}$ (resp.$e\Lambda_{M,G}$) is invariant under 
$\l \leftrightarrow \bar\l$ (resp. $\l \leftrightarrow 1/\bar\l$).
The involution $\l \leftrightarrow 1/\bar\l$ preserves the set of rays 
through zero.

On the other hand, under orientation reversal, we have 
$\Lambda_{\tau M,G}=-\Lambda_{M,G}$ and 
$\eLambda_{\tau M,G}=\tau \eLambda_{M,G}$,
where $\tau(\l)=1/\l$ for $\l \neq 0$ and $\tau(0)=0$. We thank C. Zickert
for help in identifying those involutions.

Thus, a typical picture for $\Lambda_{M,G}$ and $\eLambda_{M,G}\setminus\{0\}$
is the following:
$$
\psdraw{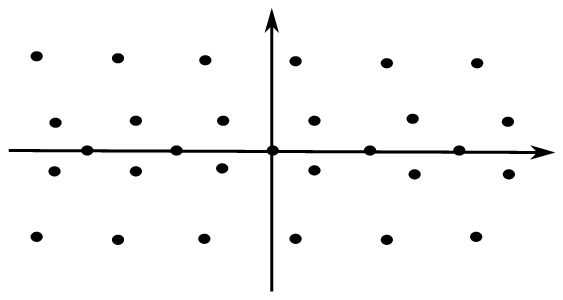}{2in} \quad 
\psdraw{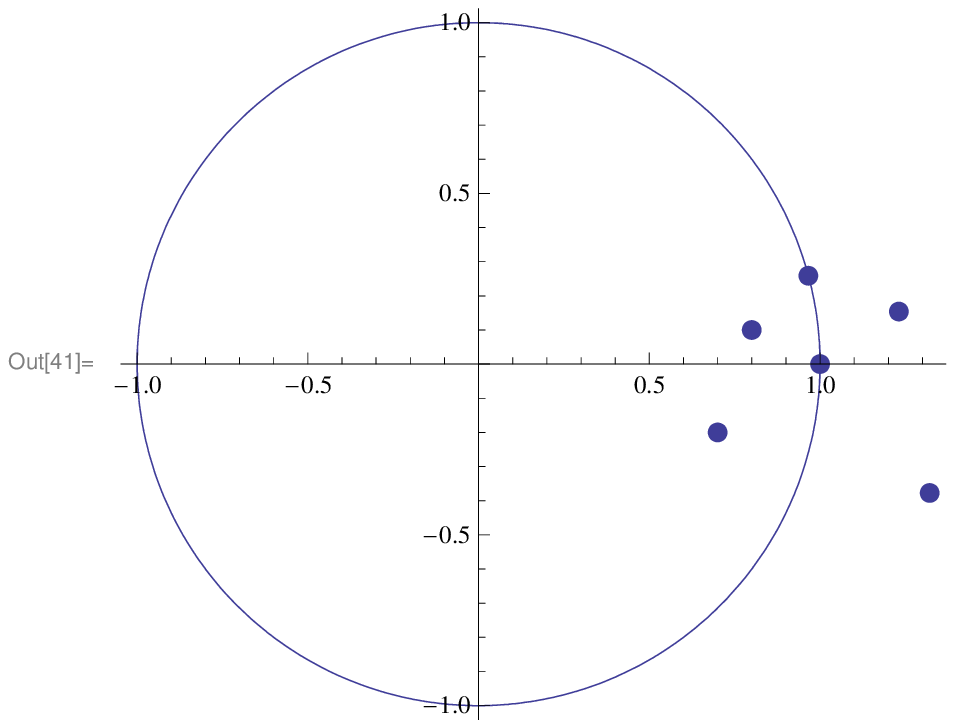}{2in}
$$
where the horizontal spacing between two dots in any horizontal line is 
$4 \pi^2=39.4784176044\dots $. 
\end{remark}

Complexification is a key idea, theoretically, as well as computationally. 
For example, 
$X_{G_{\BC}}(M)$ is an algebraic variety whereas its real part $X_{G}(M)$ is
only a compact set with little structure. The only systematic
way (known to us) to give exact formulas for the critical values of $\CS$ is 
to actually compute the critical values of $\CS_{\BC}$ and then 
decide which of these are critical values of $\CS$. For $G=\SU(2)$, there 
are exact and numerical computer implementations for the critical values
of $\CS_{\BC}$: see {\tt snap} \cite{Sn} and \cite{Ne, DZ}. 

Complexification also reveals the arithmetic structure of $\La_{M,G}$:
its elements are {\em periods} of weight $2$
(in the sense of Kontsevich-Zagier \cite{KZ}), of a 
rather special kind. Namely, the critical values of $\CS_{\BC}$
are $\BQ$-linear combinations of the {\em Rogers dilogarithm} function
evaluated at algebraic numbers. The latter is defined by:

\begin{equation}
\lbl{eq.rogers}
L(z)=\Li_2(z) +\frac{1}{2} \log(z) \log(1-z)-\frac{\pi^2}{6} 
\end{equation} 
for $z \in (0,1)$ and analytically continued as a multivalued analytic 
function in $\BC\setminus\{0,1\}$. Here,  $\Li_2(z)=\sum_{n=1}^\infty z^n/n^2$ 
is the classical dilogarithm function. For $G=\SU(2)$, our identification
of the complexified Chern-Simons action with \cite{NZ,Ne,GZ} is as follows:

\begin{equation}
\lbl{eq.CCS}
\CS_{\BC}(\rho)=i \text{Vol}(\rho) + \CS(\rho).
\end{equation}

For higher rank groups, exact 
formulas for the critical values of $\CS_{\BC}$ may also be given in terms
of the Rogers dilogarithm function at algebraic numbers. 
This will be explained in detail in a separate publication.
As an illustration of the above discussion let us give an example.

\begin{example}
\lbl{ex.1}
If $M$ is obtained by $1/2$ surgery on the $4_1$ knot, then 
$\eLambda\setminus\{0\}$ consists of $13$ points plotted as follows:
$$
\psdraw{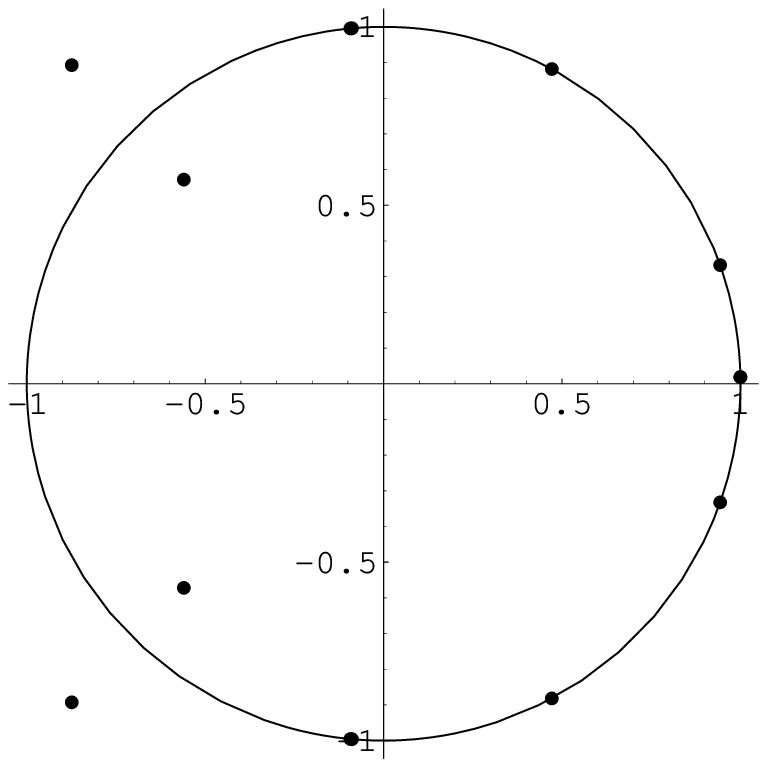}{2in}
$$
In this picture, a higher resolution reveals that the points nearest to the
vertical axis consist of {\em two} distinct but close pairs. 
We thank C. Zickert for providing us with an exact and numerical computation 
of the critical values of the complexified Chern-Simons map. 
\end{example}

Let us end this section with a problem:

\begin{problem}
\lbl{prob.1}
Give an direct relation between the cubic polynomial action 
\eqref{eq.CSA} and the Rogers dilogarithm \eqref{eq.rogers}.
\end{problem}

A transcendental relation between the Chern-Simons action and the Rogers
dilogarithm was given in \cite[Sec.6.2]{Ga4}, using the third algebraic 
$K$-theory group $K_3^{\text{ind}}(\BC)$.

\subsection{Extension to knots in 3-space}
\lbl{sub.knots}

So far, the discussion involved closed 3-manifolds. Let us now consider 
knots $K$ in the 3-sphere. For simplicity, we will
assume that $G=\SU(2)$ (so $G_{\BC}=\SL(2,\BC)$) in this section.

Let us fix a knot $K$ in 3-sphere and a nonnegative integer $n$. 
Let us denote by $Z_{K,\SU(2),n}$ the {\em Kashaev invariant}
of $K$ (see \cite{Ka}), which is also identified by \cite{MM} with the value 
of the Jones polynomial colored by the $n$-th dimensional irreducible 
representation of $\mathfrak{sl}_2$ evaluated at $q=e^{2 \pi i/n}$ 
(and normalized to be $1$ at the unknot).
Thus, we may define:

\begin{equation}   
\lbl{eq.Fzknot}
\Lnp_{K,\SU(2)}(z)=\sum_{n=0}^\infty Z_{K,\SU(2),n} z^n
\end{equation}

We define the perturbative invariant $\Lp_{K,\SU(2)}(z)$ as follows:
\begin{itemize}
\item[(a)]
Take the sequence $J_{K,n}(q) \in \BZ[q^{\pm}]$ of the Jones polynomials
of $n$, colored by the $n$-dimensional 
irreducible representation of $\mathfrak{sl}_2$, and normalized by
$J_{\text{unknot},n}(q)=1$. See for example, \cite{Tu1,Tu2}.
\item[(b)]
It turns out that there exists a power series $J_K(u,q) \in \BQ(u)[[q-1]]$
so that $J_K(q^n,q)=J_{K,n}(q) \in \BQ[[q-1]]$ for all $n \in \BN$;
see for example \cite{Ga2,GL3}.
\item[(c)]
Consider the power series $J_K(1,e^{1/x})\in \BQ[[1/x]]$.
\item[(d)]
Define $\Lp_{K,\SU(2)}(z)=\calB(J_K(1,e^{1/x})) \in \BQ[[z]]$.
\end{itemize}
In \cite{GL3} (resp. \cite{GL2}) it was shown that $\Lnp_{K,\SU(2)}(z)$ (resp. 
$\Lp_{K,\SU(2)}(z)$) is analytic for $z$ in a neighborhood of $0$.
Regarding the critical values of the Chern-Simons action, we will 
consider only {\em parabolic} $\SL(2,\BC)$ representations; i.e.,
those representations so that the trace of every peripheral element is $\pm 2$.
As in the case of closed 3-manifolds, we may identify the moduli space
of {\em parabolic} flat $\SL(2,\BC)$-connections on the knot complement
with $X_{\SL(2,\BC)}^{\mathrm{par}}(K)$:

\begin{equation}
\lbl{eq.XGpar}
X_{\SL(2,\BC)}^{\mathrm{par}}(K)=\Hom^{\mathrm{par}}(\pi_1(S^3\setminus K),
\SL(2,\BC))/\SL(2,\BC)
\end{equation}
In addition, we have a map, described in detail in \cite{GZ}:

\begin{equation}
\lbl{eq.CSCknot}
\CS_{\BC}: X_{\SL(2,\BC)}^{\mathrm{par}}(M) \longto \BC/\BZ(2)
\end{equation}
$\CS_{\BC}$ is again a locally constant map, and takes finitely many values
in $\BC/\BZ(2)$, and allows us to define the sets $\Lambda_{K,\SU(2)}$
and $\eLambda_{K,\SU(2)}$.

Let us end this section with an example of the simplest knot $3_1$
and the simplest hyperbolic knot $4_1$.

\begin{example}
\lbl{ex.31}
If $K=3_1$ is the right hand trefoil knot $3_1$, then
\begin{equation}
\lbl{eq.elambda31}
\eLambda=\{e^{\pi i/12},1,0\}
\end{equation}
can be plotted as follows:
$$
\psdraw{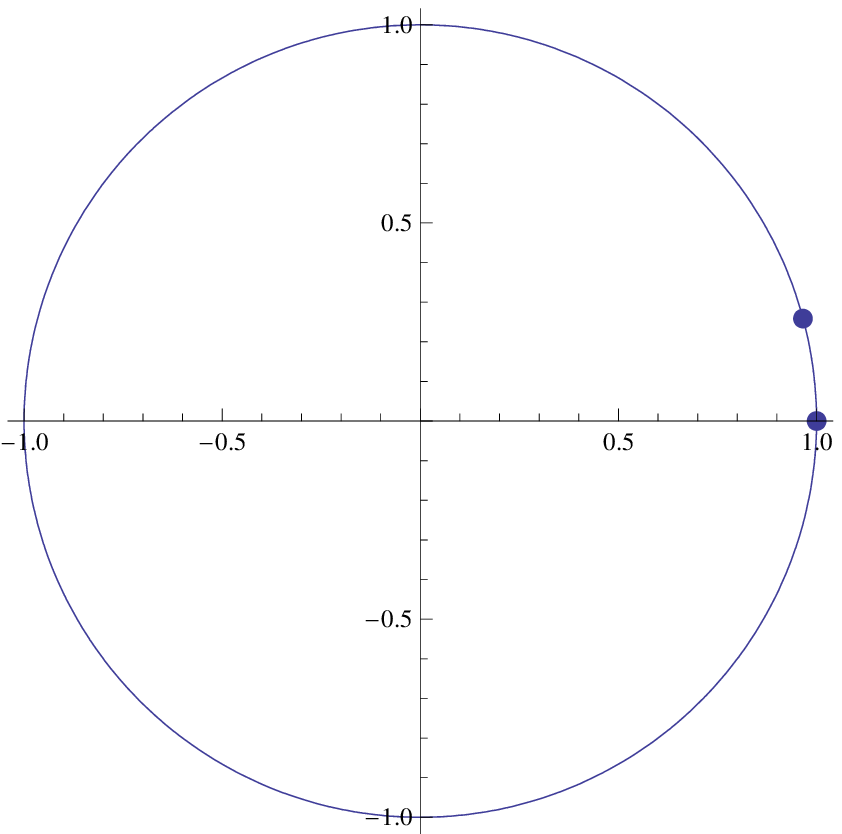}{1.7in}
$$
\end{example}

\begin{example}
\lbl{ex.41}
If $K=4_1$ is the simplest hyperbolic knot, then
\begin{equation}
\lbl{eq.elambda41}
\eLambda=\{e^{-\Vol(4_1)/(2\pi)},1, e^{\Vol(4_1)/(2\pi)},0\}
\end{equation}
can be plotted as follows:

$$
\psdraw{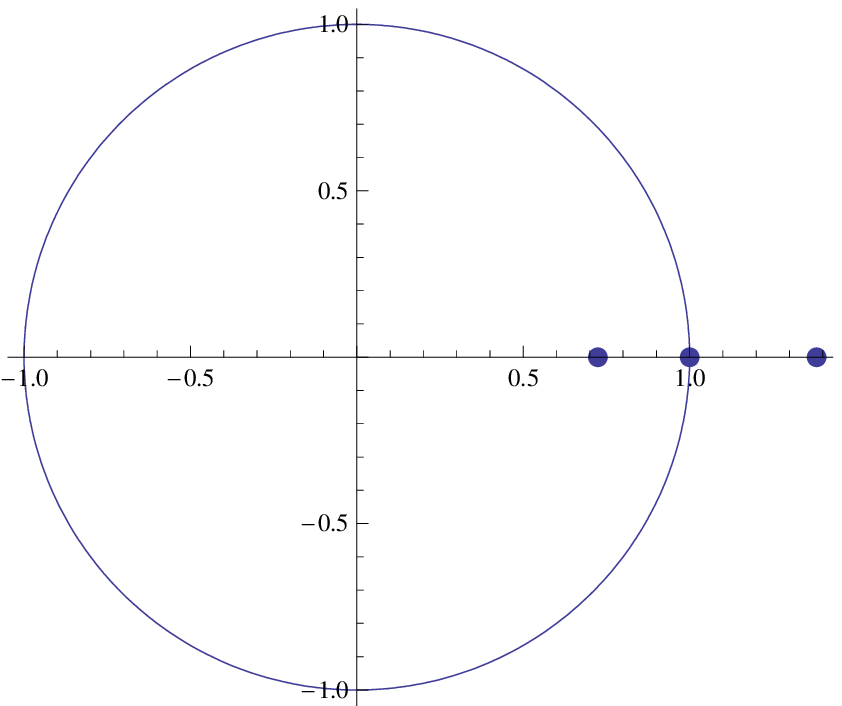}{2in}
$$
\noindent
where 
$$
\Vol(4_1)=-i \Li_2(e^{2 \pi i/6})+i\Li_2(e^{-2 \pi i/6}) = 
2.02988321281930725004240510855 \dots
$$
is the {\em Volume} of $4_1$; see \cite{Th}, and numerically,

\begin{eqnarray*}
e^{-\Vol(4_1)/(2\pi)} &=& 0.72392611187952434703122933736\dots \\
e^{\Vol(4_1)/(2\pi)} &=& 1.38135644451849779337146695685\dots
\end{eqnarray*}
\end{example}

\section{Arithmetic resurgent functions}
\lbl{sec.ar}

\subsection{Resurgent functions}
\lbl{sub.resfunctions}

The arithmetic nature of $\La_{M,G}$ is only the beginning. 
It turns out that 

\begin{itemize}
\item[(a)]
the {\em Ray-Singer torsion invariants} associated to a 
$G_{\BC}$-representation of $\pi_1(M)$ are also algebraic numbers 
(this is proven and discussed in detail in \cite{DG}), 
\item[(b)]
in case $M$ is hyperbolic and $G=\SU(2)$ the geometric representation
is defined over a number field; see \cite{MR},
\item[(c)]
the perturbative expansions of the quantum invariants
$Z_{M,G,n}$ around an $G_{\BC}$-representation of $\pi_1(M)$ are conjectured
to be algebraic numbers; see for example \cite{GM} and \cite{CG1}.
\end{itemize}
 
The need to formulate these algebricity properties in a uniform way,
as well as some results in some key cases, lead us to the notion of
an arithmetic resurgent series, which is the focus of this section.

Along the way, we will also discuss the auxiliary notion of a Gevrey series 
of mixed type, perhaps of interest on its own.

{\em Resurgence} was coined by \'Ecalle in his study of analytic continuation
of formal and actual solutions of differential equations, linear or not;
see \cite{Ec1}. An earlier term used by \'Ecalle was
the notion of {\em endless analytic continuation}. 
The concept of resurgence has influenced our thinking deeply. Unfortunately,
it is not easy to find an accepted definition of resurgence, or a reference
for it in the literature. On the other hand, there are several expositions
of instances of resurgence, covering special cases of this rather general 
notion. The curious reader may consult \cite{Co1,Co2,Dl,Ml,Sa} for a 
detailed discussion in addition to the original work \cite{Ec1}.

Given the gap in the literature, we will do our best to give a working
(and exact) definition of resurgence, which features some properties 
which are arithmetically important, and analytically rare.
Let us begin by recalling the monodromy of multivalued germs of interest
in this paper. We refer the reader to \cite{Ml} for further details.
The next definition is motivated by the types of singularities that appear
in algebraic geometry; see \cite{Kz}.

\begin{definition}
\lbl{def.quasigerm}
A {\em multivalued analytic germ} $f(z)$ at $z=0$ is called 
{\em quasi-unipotent} if
its monodromy $T$ around $0$ satisfies the condition:
\begin{equation}
\lbl{eq.Tmonodromy}
(T^r-1)^s=0
\end{equation} 
for some nonzero natural numbers $r$ and $s$. 
\end{definition}

It is easy to see that a quasi-unipotent germ $f(z)$ can be written as a 
finite sum of germs of the form:
\begin{equation}
\lbl{eq.nilsson}
\sum_{\a,\b} c_{\a,\b} z^{\a} (\log(z))^{\b} h_{\a,\b}(z)
\end{equation}
where $\a \in \BQ$, $\b \in \BN$, and $h_{\a,\b}(z) \in \BC\{z\}_0$, where
$\BC\{z\}_0$ is the ring of power series convergent at $z=0$ (identified
with the ring of germs of functions analytic at $z=0$). 
See for example, \cite{Ml}. Series of the form \eqref{eq.nilsson} are often
known in the literature as series of the {\em Nilsson class}; see 
\cite{Ni1,Ni2}.

The rationality of the exponents $\{\a\}$ above is an important feature that
always appears in algebraic geometry and arithmetic and rarely appears in
analysis. For a further discussion; see \cite{Ga5}.

\begin{definition}
\lbl{def.res}
We say that $G(z)=\sum_{n=0}^\infty a_n z^n$ is an {\em resurgent 
series} (and write $G(z) \in \RES$) if
\begin{itemize}
\item[(a)]
$G(z)$ is convergent at $z=0$.
\item[(b)]
$G(z)$ has analytic continuation as a multivalued function in $\BC\setminus
\La$, where $\La \subset \BC$ is a discrete subset of $\BC$.
\item[(c)] The local monodromy is quasi-unipotent. 
\end{itemize}
\end{definition}

In what follows, we will make little distinction between a germ, its analytic
continuation, and the corresponding function. So, we will speak about the
algebra of resurgent functions.

\subsection{Gevrey series of mixed type}
\lbl{sub.GM}

The following definition is motivated by the properties of some
power series that are associated to knotted 3-dimensional objects.

\begin{definition}
\lbl{def.GM}
We say that series $G(z)=\sum_{n=0}^\infty a_n z^n$ is a {\em Gevrey series of
mixed type} $(r,s)$ if 
\begin{itemize}
\item[(a)]
$r,s \in \BQ$ and the coefficients $a_n$ lie in a number field $K$, and
\item[(b)]
there exists a constant $C>0$ so that for every $n \in \BN$
the absolute value of every Galois conjugate of $a_n$ is less than or equal to 
$C^n n!^r$, and
\item[(c)]
the common denominator of $a_0/0!^s,\dots, a_n/n!^s$ is less than or equal 
to $C^n$.
\end{itemize}
\end{definition}

\begin{remark}
\lbl{rem.GM1}
If $G(z)$ is Gevrey of mixed type $(r,s)$ and $r' \geq r$, $s' \leq s$,
then $G(z)$ is also Gevrey of mixed type $(r',s')$.
\end{remark}

Let $\BC^{\GM}\{z\}$ (resp. $\BC^{\GM}\{z\}_{r,s}$) denote the 
$\overline{\BQ}[z]$-algebra of Gevrey series of mixed type (resp. mixed type
$(r,s)$).

\begin{remark}
\lbl{rem.GM2}
The Gevrey series of mixed type $(r,r)$ are precisely the important class of 
{\em arithmetic Gevrey series of type} $r$, introduced and studied by 
Andr\'e; see \cite{An}.
\end{remark}

\begin{remark}
\lbl{rem.GM3}
A $G$-{\em function} $G(z)$ in the sense of Siegel is a Gevrey series of mixed
type $(0,0)$ which is {\em holonomic}, i.e., it satisfies a linear ODE
with coefficients in $\overline{\BQ}[z]$; see \cite{An,Bo,DGS}.
\end{remark}

\subsection{Arithmetic resurgent functions}
\lbl{sub.Gfunctions}

Restricting the functions $h_{\a,\b}(z)$ in \eqref{eq.nilsson} to be 
Gevrey series of type $(0,s)$, we arrive at the notion of an arithmetic 
quasi-unipotent germ.

\begin{definition}
\lbl{eq.aquasigerm}
We say that a multivalued analytic germ $f(z)$ at $z=0$ is {\em arithmetic 
quasi-unipotent} if it can be written as a finite sum of 
germs of the form
\eqref{eq.nilsson} where $h_{\a,\b}(z) \in \BC^{\GM}\{z\}$.
\end{definition}

Combining this definition with the notion of a resurgent function,
we arrive at the notion of an arithmetic resurgent function.
Let $\calP \subset \BC$ denote the countable set of {\em periods} in the
sense of Kontsevich-Zagier; see \cite{KZ}.

\begin{definition}
\lbl{def.ares}
We say that $G(z)=\sum_{n=0}^\infty a_n z^n$ is an {\em arithmetic resurgent 
series} (and write $G(z) \in \ARES$) if
\begin{itemize}
\item[(a)]
$G(z)$ is a resurgent series.
\item[(b)]
The singularities $\La$ of $G(z)$
is a discrete subset of $\calP$, where $\calP$ denotes the set of 
{\em periods} as defined by Kontsevich-Zagier; see \cite{KZ}.
\item[(c)]
The local monodromy is arithmetic quasi-unipotent.
\item[(d)]
The global monodromy is defined over $\bar\BQ$.
\end{itemize}
\end{definition}

\subsection{The Taylor series of an arithmetic resurgent function}
\lbl{sub.propar}

In a separate publication we will give the proof of the following
proposition which shows that the coefficients of the Taylor series at the
origin of an arithmetic resurgent function have asymptotic expansions
themselves. The transseries conclusion of the proposition below (without
any claims on the mixed Gevrey type) follows from the resurgence hypothesis
on $G(z)$ alone, and are studied systematically in the upcoming book
of Costin \cite{Co2}, as well as in \cite{Co3}.

\begin{proposition}
\lbl{prop.expar}
If $G(z)=\sum_{n=0}^\infty a_n z^n$ is arithmetic resurgent, then
$$
a_n \sim \sum_{\l} \l^{-n} f_{\l}\left(\frac{1}{n}\right)
$$
where the sum is over the finite set of singularities of $G(z)$ nearest
to the origin, and 
$f_{\l}(z)$ is a finite sum of series of the form \eqref{eq.nilsson}
where $h_{\a,\b,\l}(z)$ are Gevrey of mixed type $(1,s)$ (for some $s$).
\end{proposition}

\subsection{Arithmetic invariants of arithmetic resurgent functions} 
\lbl{sub.invariantsar}

Obvious arithmetic invariants of an arithmetic resurgent function $G(z)$
are:

\begin{itemize}
\item[(a)] 
The set of singularities $\La \subset \calP$.
\item[(b)]
The local quasi-unipotent monodromy, and its field of definition. 
\item[(c)]
The global monodromy, defined over $\bar\BQ$.
\end{itemize}

\subsection{$G$-functions are arithmetic resurgent}
\lbl{sub.3source}

This section is logically independent of the rest of the paper and can be
skipped at first reading, although
it provides some useful examples of arithmetic resurgent functions.
A main example of arithmetic resurgent functions comes from 
the following theorem of Andr\'e.

\begin{theorem}
\lbl{thm.andre}
$G$-functions are arithmetic resurgent with singularities
a finite set of algebraic numbers.
\end{theorem}

$G$-functions arise naturally in three contexts:

\begin{itemize}
\item[(a)]
From geometry, related to the regularity of the Gauss-Manin connection. 
\item[(b)]
From arithmetic. 
\item[(c)]
From enumerative combinatorics.
\end{itemize}

For a geometric construction of resurgent functions, let us recall the
following result from \cite{Kz}; see also \cite{De,Br}.
Let $\overline{S}/\BC$ be a projective non-singular connected curve
and $S=\overline{S}\setminus\{p_1,\dots,p_r\}$ the complement of a finite set 
of points. Suppose that
$$
\pi: X \longto S
$$
is a proper and smooth morphism. For every $i$, the algebraic de Rham 
cohomology $H^i_{\mathrm{dR}}(X/S)$ is an algebraic differential equation on 
$S$,
and the local system $H^i(X_s,\BC)$ (for $s \in S$)
is the local system of germs of solutions
of that equation. Let $T$ denote the local monodromy around a point $p_i$.
Then, we have the following theorem.

\begin{theorem}
\lbl{thm.katz}\cite{Ka}
The algebraic differential equation is regular singular and the local
monodromy $T$ is quasi-unipotent.
\end{theorem}

The $G$-functions obtained by Theorems \ref{thm.katz} and 
\ref{thm.andre} are closely related. The main conjecture is that all 
$G$-functions come from geometry. For a discussion of this topic, 
and for a precise formulation of the Bombieri-Dwork Conjecture, 
see the survey papers of \cite{Bo,Ka} and also \cite[p.8]{To}.

Let us discuss a third source of resurgent functions, which was 
discovered recently by the author in \cite{Ga5}. 

\begin{definition}
\lbl{def.hyperg}
A {\em hypergeometric term} $\ft_{n,k}$ (in short, {\em \gterm }) 
in variables $(n,k)$ where $k=(k_1,\dots,k_r)$
is an expression of the form:
\begin{equation}
\lbl{eq.defterm}
\ft_{n,k}=C_0^n \prod_{i=1}^r C_i^{k_i} \prod_{j=1}^J A_j(n,k)!^{\e_j}
\end{equation}
where $C_i \in \overline{\BQ}$ for $i=0,\dots,r$, 
$\e_j=\pm 1$ for $j=1,\dots,J$, and $A_j$ are integral 
linear forms in the variables $(n,k)$. We assume that for every 
$n \in \BN$, the set 
\begin{equation}
\lbl{eq.kset}
\{ k \in \BZ^r \, | \, A_j(n,k) \geq 0, \,\, j=1,\dots, J \}
\end{equation}
is finite. We will call a \gterm\ {\em balanced} if in addition it
satisfies the {\em balance condition}:
\begin{equation}
\lbl{eq.Ajsum}
\sum_{j=1}^J \e_j A_j=0.
\end{equation}
\end{definition}

Given a \bterm\ $\ft$, consider the corresponding sequence $(a_{\ft,n})$
defined by 
\begin{equation}
\lbl{eq.an}
a_{\ft,n}=\sum_k \ft_{n,k}
\end{equation}
where the summation index lies in the finite set \eqref{eq.kset},
and the corresponding generating series:
\begin{equation}
\lbl{eq.Gz}
G_{\ft}(z)=\sum_{n=0}^\infty a_{\ft,n} z^n \in \BQ[[z]].
\end{equation}
We will call sequences of the form $(a_{\ft,n})$ {\em balanced multisum 
sequences}. 

\begin{theorem}
\lbl{thm.Ga}\cite{Ga3}
For every \bterm\ $\ft$, the generating series $G_{\ft}(z)$ is a $G$-function. 
\end{theorem}

Let us point out that the proof of Theorem \ref{thm.Ga} in general
offers no help of locating the singularities of the function $G_{\ft}(z)$. 
To fill this gap, the author developed an efficient ansatz for the location 
of the singularities of $G_{\ft}(z)$; see \cite{Ga3}. 

\section{An arithmetic resurgence conjecture}
\lbl{sec.conj}


A {\em knotted object} $\calK$  denotes 
either a closed 3-manifold $M$ or a knot $K$ in 3-space. A {\em pair}
$(\calK,G)$ denotes either a closed 3-manifold $M$ and a compact Lie group 
$G$, or a knot $\calK=K$ in 3-space and $G=\SU(2)$.
If $G=\SU(2)$ and $M$ or $K$ is hyperbolic, let 

\begin{equation}
\lbl{eq.Lambdageom}
\Lambda^{\mathrm{geom}}_{\calK,G}=\cup_{\rho}
(-\CS_{\BC}(\rho)+\BZ(2))
\end{equation}
denote the critical values of the Galois conjugates $\rho$ of the 
geometric $\SL(2,\BC)$-representation.

We now have all the ingredients to formulate our Arithmetic 
Resurgence Conjecture, which is a refinement of the Analytic Continuation
Conjecture \ref{conj.0}.

\begin{conjecture}
\lbl{conj.1}(Arithmetic Resurgence)
For every pair $(\calK,G)$, $\Lnp_{\calK,G}(e^{z/(2 \pi i)})$ 
and $\Lp_{\calK,G}(z)$ are arithmetic resurgent  with possible singularities 
in the set $\Lambda_{\calK,G}$. If $\calK$ is hyperbolic and $G=\SU(2)$, then
the singularities of the above functions include 
$\Lambda^{\mathrm{geom}}_{\calK,G}$. 
\end{conjecture}


Conjecture \ref{conj.1} implies the following corollaries.

\begin{corollary}
\lbl{cor.11}
For every pair $(\calK,G)$, 
the power series $\Lnp_{\calK,G}(z)$ has analytic continuation as a multivalued
function on the complement $\BC\setminus\eLambda_{\calK,G}$ of the finite set 
$\eLambda_{\calK,G}$.
\end{corollary}

\begin{corollary}
\lbl{cor.12}
If $M$ is a closed hyperbolic 3-manifold, then the Witten-Reshetikhin-Turaev
invariants determine the Volume of $M$ by:r
\begin{equation}
\lbl{eq.RT}
e^{-\Vol(M)/(2 \pi)}=\min\{|\l| \,\, | \,\, \Lnp_{M,\SU(2)}(z) 
\,\, \text{is singular at}\,\, z=\l \neq 0 \}. 
\end{equation}
This follows from the fact that $\Lnp_{M,\SU(2)}(z)$ has a singularity
at 
$$
e^{-(\CS_{\BC}(\rho)+\BZ(2))/(2 \pi i)}=
e^{-\Vol(\rho)/(2 \pi) + i \theta_{\rho}}
$$
and $\Vol(\rho) \leq \Vol(\rho_M)=\Vol(M)$ where $\rho_M$ is a discrete 
faithful representation.
\end{corollary}

\begin{corollary}
\lbl{cor.13}
Witten's conjecture (formulated in \cite{Wi})
regarding the asymptotic expansion of the Witten-Reshetikhin-Turaev 
invariants holds.
\end{corollary}

\begin{corollary}
\lbl{cor.14}
For every hyperbolic knot $K$ in 3-space, the Kashaev invariants determine
the Volume of $K$ by:
\begin{equation}
\lbl{eq.VC}
e^{-\Vol(K)/(2 \pi)}=\min\{|\l| \,\, | \,\, \Lnp_{K,\SU(2)}(z) \,\, 
\text{is singular at}\,\, z=\l \neq 0 \}. 
\end{equation}
Moreover, there is an asymptotic expansion of the Kashaev invariants in
powers of $1/n$ using Proposition \ref{prop.expar}.
\end{corollary}



\section{The Habiro ring, and P versus NP}
\lbl{sec.habiro}

In this section we we describe an arithmetic relation, due to Habiro, 
between the perturbative $\Lp_{\calK,G}(z)$ and the non-perturbative 
$\Lnp_{\calK,G}(z)$ invariants of knotted objects.
This section is independent of our conjecture. However, Habiro's results 

\begin{itemize}
\item[(a)]
are a good complement of our conjecture,
\item[(b)]
are important and interesting on their own right,
\item[(c)]
point to a different arithmetic origin for the invariants of knotted
objects. This point of view has been studied by Gukov-Zagier \cite{GZa}.
\end{itemize}

In this section, a knotted object $\calK$ denotes either 
a homology sphere $M$ or a knot $K$. For simplicity, we will assume that
$G=\SU(2)$. A theorem of Habiro implies that 
$\Lnp_{\calK,\SU(2)}(z)$ determines $\Lp_{\calK,\SU(2)}(z)$ and vice-versa;
\cite{Ha1,Ha2}. Let us explain more about Habiro's key results.
In \cite{Ha2} Habiro introduces the ring

\begin{equation}
\lbl{eq.habiro}
\widehat{\BZ[q^{\pm}]}=\lim_{\leftarrow n} \BZ[q^{\pm}]/((q)_n)
\end{equation}
where $(q)_n$ is the quantum $n$-factorial defined by:

\begin{equation}
\lbl{eq.qfactorial}
(q)_n=\prod_{k=1}^n (1-q^k)
\end{equation}
with $(q)_0=1$. In a sense, one may think of elements of the Habiro ring
as complex-valued {\em analytic functions} with domain $\Omega$, the set
of complex roots of unity. This way of thinking is motivated by the following
features of the Habiro ring, shown in \cite{Ha1,Ha2}:

\begin{itemize}
\item[(a)]
It is easy to see that every element $f(q) \in \widehat{\BZ[q^{\pm}]}$ can be
written (nonuniquely) in the form:
\begin{equation}
\lbl{eq.h2}
f(q)=\sum_{n=0}^\infty f_n(q) (q)_n, \qquad  
 f_n(q) \in
\BZ[q^{\pm}], \quad \text{for} \quad n \in \BN.
\end{equation}
Note however that the above form is not unique, since for example:
$$
1=\sum_{n=0}^\infty q^{n+1} (q)_n.
$$
Nevertheless the form \eqref{eq.h2} can be used to generate easily elements
of the Habiro ring.
\item[(b)] Elements of the Habiro ring can be evaluated at complex roots of
unity. In other words, there is a map:
\begin{equation}
\lbl{eq.habironp2}
\widehat{\BZ[q^{\pm}]} \longto \BC^{\Om}, \qquad
f(q) \longto \left( f: \Omega \longto \BC, \qquad \w \mapsto f(\w)\right).
\end{equation}
In particular, we can associate a map:
\begin{equation}
\lbl{eq.habironp}
\widehat{\BZ[q^{\pm}]} \longto \BC[[z]], \qquad
f(q) \longto \Lnp_f(z)=1+\sum_{n=1}^\infty f(e^{2 \pi i/n}) z^n.
\end{equation}
\item[(c)] Elements of the Habiro ring have Taylor series expansions around
$q=1$ (and also around every complex root of unity). In other words, we can
define a map:
\begin{equation}
\lbl{eq.habirop}
\widehat{\BZ[q^{\pm}]} \longto \BQ[[z]], \qquad
f(q) \longto \Lp_f(z)=\calB( f(e^{1/x})).
\end{equation}
\item[(d)] As in the case of analytic functions, the maps \eqref{eq.habironp}
and \eqref{eq.habirop} are 1-1. Thus, $\Lnp_f(z)$ determines $\Lp_f(z)$ and 
vice-versa. However, we need all the coefficients of the power series 
$\Lnp_f(z)$ to determine a single (eg. the third) coefficient of $\Lp_f(z)$.
\item[(e)] Given a
homology sphere $M$, there exists an element $f_{M,\SU(2)}(q) \in
\widehat{\BZ[q^{\pm}]}$ such that its image under the maps \eqref{eq.habironp}
and \eqref{eq.habirop} coincide with the non-perturbative and perturbative
invariants of $M$ discussed in Section \ref{sub.npinvariants} and 
\ref{sub.pinvariants}. This was a main motivation for Habiro, and 
was extended to knots in 3-space by Huynh-Le in \cite{HL}. 
\end{itemize}

One may ask for an extension of Conjecture \ref{conj.1} 
for the series $\Lnp_f(z)$ and $\Lp_f(z)$ that come from the Habiro ring. 
Unfortunately, the Habiro ring is uncountable (whereas all quantum invariants
of knotted objects lie in a countable subring) and it has little 
structure as such. Thus, it is unlikely that the series $\Lnp_f(z)$
associated to a random sequence of Laurent polynomials $(f_n(q))$ (as in 
\eqref{eq.h2}) will be resurgent. Concretely, we can pose the following
problem with overwhelming numerical evidence:

\begin{problem}
\lbl{prob.habiro}
Show that $\Lnp_f(z)$ is not resurgent when
\begin{equation}
\lbl{eq.frandom}
f(q)=\sum_{n=0}^\infty q^{2^n} (q)_n.
\end{equation}
\end{problem}

In \cite{GL4}, Le and the author introduced a countable subring
$\widehat{\BZ[q^{\pm}]}^{\mathrm{hol}}$ that consists of elements of the form:

\begin{equation}
\lbl{eq.habiroh}
f(q)=\sum_{n=0}^\infty f_n(q) (q)_n, \qquad  
 f_n(q) \in
\BZ[q^{\pm}], \quad (f_n(q)) \quad \text{is $q$-holonomic}
\end{equation}
where $q$-holonomic means that $(f_n(q))$ satisfies a linear $q$-difference
equation with coefficients in $\BQ[q^{\pm},q^{\pm n}]$; see \cite{WZ}.
In \cite{GL4} it was shown that the elements $f_{M,\SU(2)}(q)$ and 
$f_{K,\SU(2)}(q)$ of the Habiro ring actually lie in the countable subring 
$\widehat{\BZ[q^{\pm}]}^{\mathrm{hol}}$.

\begin{problem}
\lbl{prob.hab}
Show that for every $f \in  \widehat{\BZ[q^{\pm}]}^{\mathrm{hol}}$, the series
$\Lnp_f(z)$ and $\Lnp_f(z)$ are arithmetic resurgent.
\end{problem}

In the next section, we will discuss formulate a resurgence conjecture for
some special elements of the Habiro ring that do not always 
come from topology.

\section{Series of Sum-Product type}
\lbl{sec.SP}

Conjecture \ref{conj.1} and Problem \ref{prob.hab} ask for proving that 
certain power series are arithmetic resurgent. However, they do not
explain the source of resurgence. Usually, resurgence is associated with
a differential equation, linear or not; see for example 
\cite{Ec1} and also \cite{Co1,Sa}.

In this section we will give another construction of powers series
$\Lnp(z)$ and $\Lp(z)$ which aims to explain the origin of arithmetic
resurgence. This section was motivated by conversations with O. Costin
and J. \'Ecalle whom we thank for their generous sharing of their ideas.

Let us first introduce the notion of {\em series of Sum-Product type}.

\begin{definition}
\lbl{def.SPtype}
Consider function $F$ analytic in $[0,1]$ with $F(0)=0$ and the corresponding
sequence and series of sum-product type:
\begin{eqnarray}
\lbl{eq.SPan}
a_n &=& \sum_{k=1}^n \prod_{j=1}^k F\left(\frac{j}{n}\right) \\
\notag
&=& F\left(\frac{1}{n}\right)+
F\left(\frac{1}{n}\right)F\left(\frac{2}{n}\right)+\dots+
F\left(\frac{1}{n}\right)F\left(\frac{2}{n}\right)\dots
F\left(\frac{n}{n}\right)
\end{eqnarray}
and the corresponding series
\begin{equation}
\lbl{eq.LnpF}
\Lnp(z) = \sum_{n=1}^\infty a_n z^n \in \BC[[z]].
\end{equation}
Since $F(0)=0$, it follows that the formal power series

\begin{equation}
\lbl{eq.SP}
\Sigma\Pi(x):=\sum_{n=1}^\infty \prod_{j=1}^n F\left(\frac{j}{x}\right) \in 
\BC[[\frac{1}{x}]]
\end{equation}
is also well-defined. Let $\Lp(z)$ denote the Borel transform:
\begin{equation}
\lbl{eq.LpF}
\Lp(z) =\calB\left( \SP(x) \right) \in 
\BC[[z]].
\end{equation}
\end{definition}

Let us now give a flavor of some results from \cite{CG2} and \cite{ES}.
In the rest of the section, let us 
consider $F$ of the following trigonometric type:

\begin{equation}
\lbl{eq.trigF}
F(x)=\phi(e^{2 \pi i x})
\end{equation}
where

\begin{equation}
\lbl{eq.phitrig}
\phi(q)=\e q^{c\frac{n(n+1)}{2}} \prod_{r=1}^\infty (1-q^r)^{c_r}
\end{equation}
where $c \in \BZ$, $\e=\pm 1$, 
$c_r \in \BN$ for all $r$, and $c_r=0$ for all but finitely many $r$.
In \cite{Ga4} we construct elements of the extended Bloch group,
given by solutions of the algebraic equations:

\begin{equation}
\lbl{eq.phi=1}
\phi(q)=1 \qquad \text{or} \qquad \phi(q)=0.
\end{equation}
The values of these elements under the Rogers dilogarithm defines a set
$\Lambda \subset \BC$, and its exponentiated cousin 
$\eLambda=\exp(\Lambda/(2 \pi i))\cup\{0\}$.

\begin{theorem}
\lbl{thm.CG}\cite{CG2,ES}
$\Lnp(z)$ and $\Lp(z)$ are arithmetic resurgent with singularities included
in $\Lambda$.
\end{theorem}

\begin{remark}
\lbl{rem.SP1}
Notice that $\Lnp(z)=\Lnp_f(z)$ and $\Lp(z)=\Lp_f(z)$ where
$$
f(q)=\sum_{n=0}^\infty \e^n q^{cn} \prod_{r=1}^\infty (q^r)^{c_r}_n
$$
is an element of the countable Habiro subring 
$\widehat{\BZ[q^{\pm}]}^{\mathrm{hol}}$, where
$$
(q^r)_n=\prod_{k=1}^n (1-q^{kr}).
$$
Thus, Theorem \ref{thm.CG} is a special case of Problem 
\ref{prob.habiro}.
\end{remark}

\begin{remark}
\lbl{rem.SP2}
Equations \eqref{eq.phi=1} appear in the {\em dilogarithm ladders} of 
Lewin and others, whose aim is to produce interesting elements of
algebraic $K$-theory. For a detailed discussion, see \cite{Le}
and \cite{Za2}.
\end{remark}

\begin{remark}
\lbl{rem.SP3}
For the simplest knot $3_1$ and the simplest hyperbolic knot $4_1$, we have:
$$
\Lnp_{3_1,\SU(2)}(z)=\Lnp_{f_{3_1}}(z), \qquad 
\Lnp_{4_1,\SU(2)}(z)=\Lnp_{f_{4_1}}(z)
$$
(and likewise, equality for the $\Lp$-series), where
\begin{eqnarray*}
f_{3_1}(q) &=& \sum_{n=0}^\infty (q)_n \\
f_{4_1}(q) &=& \sum_{n=0}^\infty (-1)^n q^{-\frac{n(n+1)}{2}} (q)^2_n
\end{eqnarray*}
are both covered by Theorem \ref{thm.CG}.
\end{remark}

\begin{remark}
\lbl{rem.SP4}
The resurgence conclusion of 
Theorem \ref{thm.CG} is valid for very general entire functions $F$,
with some mild hypothesis. For a detailed discussion, see \cite{ES} and 
\cite{CG2}.
\end{remark}

\section{Evidence}
\lbl{sec.evidence}

\subsection{Some results}
\lbl{sub.results}

Let us summarize what is known about Conjecture \ref{conj.1}.
Conjecture \ref{conj.1} is known 

\begin{itemize}
\item[(a)]
for all 3-manifolds $M$ of the form
$\S \times S^1$ where $\S$ is a closed surface and all compact groups $G$. 
Indeed, this follows from
the fact that $Z_{M,G,n}$ is a polynomial in $n$. Thus, $\Lnp_{M,G}(z)$
is a rational function of $z$ with denominator a power of $1-z$. On the
other hand, $\eLambda_{M,G}=\{0,1\}$.
\item[(b)]
for $\Lp_{M,\SU(2)}$ where $M$ is the {\em Poincare homology sphere}, or
small Seifert fibered 3-manifolds, see \cite{CG1}. In this lucky case, one
uses explicit formulas for the coefficients of $\Lp_{M,\SU(2)}(z)$ given 
Zagier (see \cite{Za1}) which allow one show 
resurgence relatively easily. 
\item[(c)]
for the simplest knot $3_1$ (and also for
$(2,p)$ torus knots); and for the simplest hyperbolic knot $4_1$;
see \cite{CG1} and \cite{CG2}. See Remark \ref{rem.SP3}.
\end{itemize}

Our sample calculations below show the importance of the fractional 
polylogarithms and their analytic continuation, studied in detail in 
\cite{CG3}.

\subsection{Conjecture \ref{conj.1} for $S^3$}
\lbl{sub.S3}

Let us confirm Conjecture \ref{conj.1} for $S^3$. For simplicity, we will
choose $G=\SU(2)$. The case of other Lie groups is similar.
The Witten-Reshetikhin-Turaev invariant is given by \cite[Eqn.2.26]{Wi}:

\begin{equation}
\lbl{eq.WRTS3a}
Z_{S^3,\SU(2),n}=\sqrt{\frac{2}{n+2}} \sin\left(\frac{\pi}{n+2}\right).
\end{equation}
Expanding the above as a convergent power series in $1/n$:

$$
Z_{S^3,\SU(2),n}= \sqrt{2} \sum_{k=0}^\infty \frac{\pi^{2k+1} (-1)^k}{(2k+1)!}
\frac{1}{(n+2)^{2k+3/2}}
$$
and using the {\em fractional polylogarithm} $L_{\a}(z)$ defined for 
$\a \in \BC$ and $|z|<1$ by the convergent series:

\begin{equation}
\lbl{eq.fracpoly}
\Li_{\a}(z)=\sum_{n=1}^\infty \frac{z^n}{n^{\a}}
\end{equation}
it follows that

\begin{equation}
\lbl{eq.WRTS3b}
\Lnp_{S^3,\SU(2)}(z)=\frac{\sqrt{2}}{z^2} 
\sum_{k=0}^\infty \frac{\pi^{2k+1} (-1)^k}{(2k+1)!} 
\left( \Li_{2k+3/2}(z)-\zeta(2k+3/2)\right).
\end{equation}
Since $\Li_{\a}(z)$ has analytic continuation as a multivalued function
in $\BC\setminus\{0,1\}$ (see \cite{CG3}), 
it follows that $\Lnp_{S^3,\SU(2)}(z)$ has 
analytic continuation on $\BC\setminus\{0,1\}$. We can further compute the
monodromy around $z=0$ and $z=1$ using \cite{CG3}.

\subsection{Conjecture \ref{conj.1} for $S^1\times\Sigma_g$}
\lbl{sub.SS3}

Let us confirm Conjecture \ref{conj.1} for 3-manifolds of the form
$S^1 \times \Sigma_g$. For simplicity, we will
choose $G=\SU(2)$. The case of other Lie groups is similar.
The Witten-Reshetikhin-Turaev invariant is given by the famous {\em Verlinde
formula} \cite{Wi,Sz}:

\begin{equation}
\lbl{eq.WRTSSa}
Z_{S^1 \times \S_g,\SU(2),n}=\sum_{j=1}^{n+1} \left(
\frac{n+2}{2 \sin^2\frac{\pi j}{n+2}}\right)^{g-1}
\end{equation}
Altough not a priori obvious, it it true that the right hand side of the above
expression is a polynomial in $n$ of degree $3g-3$. In fact, we have 
\cite[Sec.3]{Sz}:

\begin{equation}
\lbl{eq.WRTSSb}
Z_{S^1 \times \Sigma_g,\SU(2),n}=-(2n+2)^{g-1} \text{Res}\left( 
\frac{(n+2)\cot((n+2)x)}{(2\sin x)^{2g-2}},x=0\right)
\end{equation}
when $g \geq 2$. For example, we have,


\begin{eqnarray*}
Z_{S^1 \times \Sigma_0,\SU(2),n} &=& 1 \\
Z_{S^1 \times \Sigma_1,\SU(2),n} &=& n+1 \\
Z_{S^1 \times \Sigma_2,\SU(2),n} &=& \frac{n^3}{6}+n^2+\frac{11n}{6}+1
\end{eqnarray*}
It follows that $\Lnp_{S^1 \times \Sigma_g,\SU(2)}(z)$ is a rational 
function with
denominator a power of $z-1$. For example, we have:

\begin{eqnarray*}
\Lnp_{S^1 \times \Sigma_0,\SU(2)}(z) &=& \frac{1}{1-z} \\
\Lnp_{S^1 \times \Sigma_1,\SU(2)}(z) &=& \frac{1}{(1-z)^2} \\
\Lnp_{S^1 \times \Sigma_2,\SU(2)}(z) &=& \frac{1}{(1-z)^4}
\end{eqnarray*}
Since $X_G(S^1 \times \Sigma_g)$ is connected for all $g$ and $G$, it 
follows
that $\eLambda_{S^1 \times \Sigma_g,G}=\{0,1\}$ confirming Conjecture 
\ref{conj.1}.

\subsection{Conjecture \ref{conj.1} for $3_1$}
\lbl{sub.conj131}

Let us give some details about how the work of \cite{Za1} and \cite{CG1}
verify Conjecture \ref{conj.1} for the simplest $3_1$ knot. An independent
verification of the Conjecture, valid for series of sum-product type,
can be obtained by \cite{ES}.

Equation (36) of \cite{Za1} and \cite{CG1} imply that we can write:

$$
Z_{3_1,\SU(2),n}=\zeta_{24}^{-n+3} n^{3/2} + 1 + \int_0^\infty e^{-np}
G(p) dp
$$
where $\zeta_c=e^{2 \pi i c}$ and 
$G(z)$ is a multivalued analytic function (analytic at $z=0$):
\begin{equation}
\lbl{eq.borel31}
G(z)=\frac{3\pi}{2 \sqrt{2}} 
\sum_{n=1}^\infty \frac{\chi(n) n}{(-z + n^2 \pi^2/6)^{5/2}}
\end{equation}
where
$\chi(\cdot)$ denotes the {\em unique primitive character of conductor} $12$:
\begin{equation}
\lbl{eq.chi}
\chi(n)=
\begin{cases}
1 & \text{if} \,\, n \equiv 1,11 \bmod 12 \\
-1 & \text{if} \,\, n \equiv 5,7 \bmod 12 \\
0 & \text{otherwise}
\end{cases}
\end{equation}

Together with Proposition \ref{prop.expar}, this implies that the singularities
of $\Lnp_{3_1,\SU(2)}(z)$ are $\{0,1,e^{\pi i/12}\}$. Moreover, the local 
expansion
of $\Lnp_{3_1,\SU(2)}(z)$ around $z=e^{\pi i/12}$ is given by:

\begin{eqnarray*}
\Lnp_{3_1,\SU(2)}(z) &=& \sum_{n=1}^\infty \zeta_{24}^{-n+3} n^{3/2} z^n 
+h(z)\\
&=& \zeta_8 \,\, \Li_{-3/2}(\zeta_{-24} z) +h(z)
\end{eqnarray*}
where $h(z)$ is a function holomorphic at $z=0$. Since $\Li_{-3/2}(z)$
is multivalued analytic at $\BC\setminus\{0,1\}$ (see \cite{CG3}), 
this confirms Conjecture \ref{conj.1} for $3_1$. Using a Mittag-Leffler
type decomposition for the fractional polylogarithm from \cite[Eqn.13]{CG3},
we can also verify the Symmetry Conjecture for $3_1$.

\subsection{Numerical evidence for the nearest singularity}
\lbl{sub.numerical1}

Conjecture \ref{conj.1} gives an exact formula for the singularity
of $\Lp_{K,\SU(2)}(z)$ which is nearest to the origin. Notice that this 
singularity does not in general coincide with the critical value of 
$\CS_{\BC}$ 
corresponding to the discrete faithful representation. This was indeed observed
numerically  for several {\em twist knots}.   
Let $K_p$ denote the twist knot with negative clasp and $p$ full twists, where
$p \in \BZ$:

$$ 
\psdraw{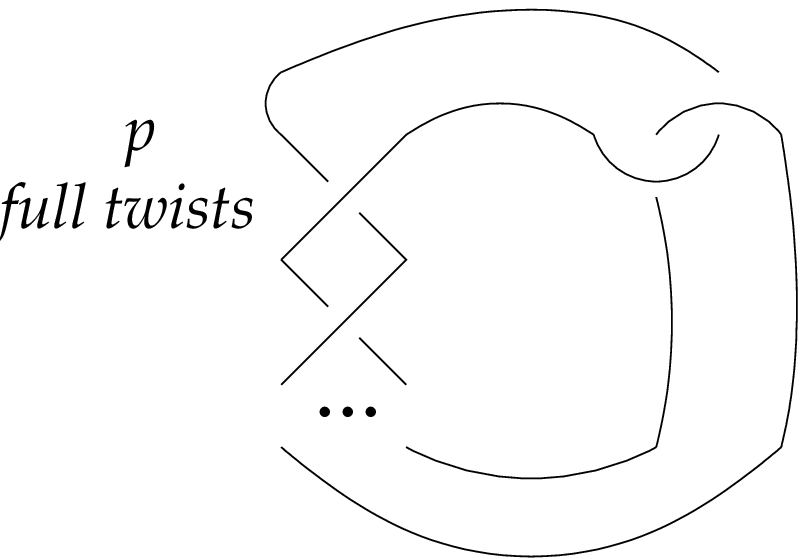}{1.3in} 
$$
In particular, we have:
\begin{equation}
\lbl{eq.twistex}
K_1=3_1, \quad K_2=5_2, \quad K_3=7_2, \quad K_4=9_2, \quad 
K_{-1}=4_1, \quad K_{-2}=6_1, \quad K_{-3}=8_1, \quad K_{-4}=10_1.
\end{equation}
The invariant trace field of $K_p$ is of type $[1,p-1]$ for $p>1$ and 
$[0,|p|]$ for $p<0$. It follows that $\Lambda_{K_p,\SU(2)}$ is a subset of
a union of $2(p-1)$ (resp. $2|p|$) horizontal lines, symmetric with respect
to the $z$-axis, and a superposition of $2$ (resp. $1$) copies of the 
$z$-axis for $p>1$ (resp. $p<0$). This set can be computed exactly and 
numerically using the methods of \cite{GZ}.

The corresponding element of the Habiro ring is given by:

\begin{equation}
\lbl{eq.phitwist}
f_{K_p,\SU(2)}(q)=\sum_{n=0}^\infty C_{K_p,n}(q) 
(q)_n (q^{-1})_n
\end{equation}
where $C_{K_p,n}(q) \in \BZ[q^{\pm}]$ denote the $n$-cyclotomic polynomial of 
$K_p$. The latter may be computed inductively with respect to $n$ for each
fixed $p$, see \cite{GS}. 

Using this formula, one can compute 500 coefficients of $\Lp_{K_p,\SU(2)}(z)=
\calB(f_{K_p,\SU(2)}(e^{1/x}))$, and then numerically compute the singularity 
of the series  $\Lp_{K_p,\SU(2)}(z)$ nearest to the origin. The numerical
method used was the following:

\begin{itemize}
\item
Fix a truncated power series
$$
L(z)=\sum_{n=0}^N a_n z^n
$$
where $N$ is a sufficiently large integer (eg $N=500$).
\item  
Using the root test, one can compute approximately the radius of convergence
$r_0$ of the series $L(z)$
\item
Plot $|L(r \exp(2 \pi i t)|$ for $r$ near the inverse of the
radius of convergence. The plot reveals a blow-up at certain values $t_0$ of
$t$. 
\item
This suggests singularities at $r_0 e^{2 \pi i t_0}$, and an asymptotic
expansion of $a_n$ with a term of the form: 
$$
r_0^{-n} e^{-2 \pi i t_0 n} n^{\alpha} \left(c_0 + c_1 \frac{1}{n} + \dots
\right)
$$
In general, we have a finite sum (over $t_0$) of terms of the above form.
\item
One can numerically compute the constants $\alpha$ and $c_0$ by fitting
data. A fitting method (also used by Zagier in \cite[p.953]{Za1}) can
improve the rate of convergence to $O(1/n^d)$ for any $d$. We used $d=100$.
\end{itemize}

This was done for the twist knots of Equation \eqref{eq.twistex}. If 
$s(K_p)$ denote the
inverse of the radius of convergence of the series $\Lp_{K_p,\SU(2)}(z)$ then
we obtain numerically that:
$$
s(K_1)=\frac{\pi^2}{6}=1.644\dots \qquad s(K_2)=1.119\dots \quad 
s(K_3)=0.882\dots \quad s(K_4)=0.745\dots \quad s(K_5)= 0.745\dots
$$
These numbers agree with the absolute value of 
$e^{\CS_{\BC}(\rho)/(2 \pi i)}$ for $\rho$ some Galois conjugate of the 
discrete faithful representation. 
We thank N. Dunfield, S. Shumakovitch and C. Zickert for their help in
the numerical computations.

Additional numerical evidence for $K_1$ and $K_{-1}$ (and for many series of 
1-dimensional sum-product type) was obtained by \cite{ES}.

\subsection{Acknowledgement}
An early version of these conjectures was delivered by lectures in Columbia
University, CUNY, Universit\'e Paris VII, and Orsay in the fall of 2006 and
in Aarhus, Baton-Rouge, Hanoi and Strasbourg in the summer of 2007. 
The author wishes to thank the organizers of the conferences for their 
hospitality, and also N. A' Campo, Y. Andr\'e, N. Dunfield,
M. Kontsevich, W. Neumann, T. Pantev, D. Sullivan, D. Thurston, 
D. Zeilberger and C. Zickert. Above all, the author expresses his 
gratitude to O. Costin, J. \'Ecalle and D. Zagier for many enlightening 
conversations.



\appendix

\section{A formal relation among $\Lnp(z)$ and $\Lp(z)$ for series of
sum-product type}

In this section we will give a formal proof of the relation among 
$\Lnp(z)$ and $\Lp(z)$ for series of sum-product type. We will use the
notation from Section \ref{sec.SP}.

Suppose that $F(x)<1$ for all $x \in (0,1]$. Then, it is easy to see that
$\Lnp(z)$ is convergent inside the unit disk $|z|<1$. The next theorem
describes an explicit relation between $\Lnp(z)$ and $\Lp(z)$.

\begin{theorem}
\lbl{thm.asF}
We have:
\begin{equation}
\lbl{eq.asF2}
\Lnp(1+z)=\log(z) \Lp(\log(1+z))+h(z)
\end{equation}
where $h(z)$ is analytic at $z=0$.
\end{theorem}

\begin{proof}
We will give only the formal calculation, leaving the analytic details
to the reader. Below, $h(z)$ will denote a germ of an analytic function
at $z=0$. For $n \in \BN$, let $c_n$
denote the coefficient of $1/x^n$ in $\SP(x)$ given in \eqref{eq.SP}. 
Let us fix $N \in \BN$ and consider $n$ large enough. Then, we have:

\begin{eqnarray*}
a_n &=& \sum_{k=1}^N \prod_{j=1}^k F\left(\frac{j}{n}\right) 
+O\left(\frac{1}{n^{N+1}}\right) \\
&=& \sum_{k=1}^N \frac{c_k}{n^k}  +O\left(\frac{1}{n^{N+1}}\right).
\end{eqnarray*}
Thus,

\begin{eqnarray*}
\Lnp(z) &=& \sum_{n=1}^\infty a_n z^n \\
&=& \sum_{n=1}^\infty \left( 
\sum_{k=1}^{N} \frac{c_k}{n^k}  +O\left(\frac{1}{n^{N+1}}\right)\right)z^n.
\end{eqnarray*}
Ignore the $O(\cdot)$ terms, and interchange summation and integration.
We obtain that

\begin{eqnarray*}
\sum_{n=1}^\infty \sum_{k=1}^N \frac{c_k}{n^k} z^n
&=& \sum_{k=1}^N c_k \sum_{n=1}^\infty  \frac{z^n}{n^k} \\
&=& \sum_{k=1}^N c_k \Li_k(z),
\end{eqnarray*}
where
$$
\Li_k(z)=\sum_{n=1}^\infty\frac{z^n}{n^k}
$$
is the $k$-polylogarithm. The latter is a multivalued analytic function
on $\BC\setminus\{0,1\}$ with an asymptotic expansion at $z=1$ of the form:
$$
\Li_k(z)=\log(z-1)\frac{\log(z)^{k-1}}{(k-1)!}+h(z)
$$
where $h(z)$ is an analytic function at $z=0$; 
see for example, \cite[Eqn.6]{Oe}. Thus,

\begin{eqnarray*}
\sum_{n=1}^\infty \sum_{k=1}^N \frac{c_k}{n^k} z^n
&=& \log(z-1) \sum_{k=1}^N c_k \frac{\log^{k-1}(z)}{(k-1)!} + h(z).
\end{eqnarray*}
Letting $N \to \infty$, and replacing $z$ by $z+1$ it follows that:

\begin{eqnarray*}
\Lnp(1+z) &=& \log(z) 
\sum_{k=1}^\infty c_k \frac{\log^{k-1}(1+z)}{(k-1)!} + h(z) \\
&=& \log(z) \Lp(\log(1+z)) + h(z).
\end{eqnarray*}
This concludes the formal calculation.
\end{proof}

\section{A path integral formula for $\Lnp(z)$}
\lbl{sec.pathintegral}

In this section we will give a path integral formula for $\Lnp_{M,G}(z)$
using as input the famous Chern-Simons path integral studied in the
seminal paper of Witten; see \cite{Wi}. With the notation of \cite{Wi}
and with our normalization we have:

$$
Z_{M,G,n}=\int_{\calA} e^{\frac{n}{2 \pi i} \CS(A)}
\mathcal{D} A 
$$
where $\calA$ is the affine space $\calA$ of $G$-connections
on the trivial bundle $M \times G$ over $M$. Since $\CS$ takes values in
$\BC/\BZ(2)$, we require that the level $n$ (which plays the role of the
inverse Planck's constant) be integer. Without loss of generality, we assume
that $n \in \BN$. 
Formally separating the $n=0$ contribution in \eqref{eq.Fz} and 
interchanging summation and integration in \eqref{eq.Fz}, it follows
that

\begin{eqnarray*}
\lbl{eq.LnpCS}
\Lnp_{M,G}(z) &=& 1+ \int_{\calA} \sum_{n=1}^\infty  
e^{\frac{n}{2 \pi i} \CS(A)} z^n dA \\
&=& 1+
z \int_{\calA} \frac{1}{e^{-\frac{n}{2 \pi i} \CS(A)}-z} \mathcal{D} A.
\end{eqnarray*}
The above formula is an infinite dimensional analogue of a 
{\em Riemann-Hilbert problem}, and was obtained during a conversation with 
Kontsevich in the fall of 2006. For a detailed discussion on the 
Riemann-Hilbert problem, see \cite{Df}. Finite dimensional analogues of
the Riemann-Hilbert problem are discussed in \cite{CG2}.

\ifx\undefined\bysame
        \newcommand{\bysame}{\leavevmode\hbox
to3em{\hrulefill}\,}
\fi

\end{document}